\documentclass[]{imsart}
\usepackage[margin=3.5cm]{geometry}
\RequirePackage{amsthm,amsmath,amsfonts,amssymb,bm,algorithm,algpseudocode}
\RequirePackage[authoryear]{natbib}
\RequirePackage[colorlinks,citecolor=green,urlcolor=green]{hyperref}
\usepackage{xcolor}
\usepackage{listings}
\RequirePackage{graphicx}
\usepackage{tikz}
\usepackage{pgfplots}
\usetikzlibrary{arrows.meta, positioning, shapes}

\startlocaldefs
\theoremstyle{plain}

\newtheorem{theorem}{Theorem}[section]

\newtheorem{conjecture}[theorem]{Conjecture}
\theoremstyle{remark}
\newtheorem{definition}[theorem]{Definition}

\endlocaldefs

\begin{document}

\begin{frontmatter}
\title{Graphical Finite Population Sampling}
%\title{A sample article title with some additional note\thanksref{t1}}
\runtitle{A Graphical Approach to Finite Population Sampling}
%\thankstext{T1}{A sample additional note to the title.}

\begin{aug}
%%%%%%%%%%%%%%%%%%%%%%%%%%%%%%%%%%%%%%%%%%%%%%%
%% Only one address is permitted per author. %%
%% Only division, organization and e-mail is %%
%% included in the address.                  %%
%% Additional information can be included in %%
\author[A]{\fnms{Bardia}~\snm{Panahbehagh}\\{panahbehagh@khu.ac.ir}\orcid{0000-0001-9122-7777}},
%%%%%%%%%%%%%%%%%%%%%%%%%%%%%%%%%%%%%%%%%%%%%%
%% Addresses                                %%
%%%%%%%%%%%%%%%%%%%%%%%%%%%%%%%%%%%%%%%%%%%%%%
\address[A]{Faculty of Mathematical Sciences and Computer, Kharazmi University,\\Tehran, Iran}

\end{aug}

\begin{abstract}
This paper introduces an innovative and intuitive finite population sampling method that has been developed using a unique graphical framework. In this approach, first-order inclusion probabilities are represented as bars on a two-dimensional graph. By manipulating the positions of these bars, researchers can create a wide range of different sampling designs. This graphical visualization of sampling designs facilitates the exploration of alternative designs and may simplify certain aspects of the implementation compared to traditional mathematical algorithms. This novel approach holds significant promise for tackling complex challenges in sampling, such as achieving an optimal design. By applying a version of the greedy best-first search algorithm to this graphical approach, the potential for integrating intelligent algorithms into finite population sampling is demonstrated. 
\end{abstract}



\end{frontmatter}

\section{Introduction}\label{Sec.Intro}
The conventional approach to finite population sampling is to deal with sets and indices that  are found in a reference book such as \citet{till:2020} as follows. 
	Let $U=\{1,\dots,N\}$ be a population of size $N$, and define
	$$
	\mathcal{S} =\{s_1,s_2,\dots,s_T\}.% \{s\subset U\},
	$$
	as the set of all possible samples or subsets.
	The aim is to select a random sample~$\mathbb{S}$ from $\mathcal{S}$ with first-order inclusion probabilities (FIP) $\bm{\pi}=\{\pi_k,k\in U\},$ and size $n_\mathbb{S}$
	where
	$$
	\sum_{k\in U}\pi_k =E(n_\mathbb{S}).
	$$
Let $\bm{p}=\{p_t=Pr(\mathbb{S}=s_t);t=1,2,\dots,T\}$ be a sampling design that assigns a probability to each possible sample such that
	\begin{equation}
		p_t>0,\;\;\;\sum_{t=1}^T p_t =1, \mbox{ and }
  %\sum_{\scriptstyle t=1,2,\dots,T  \atop  \scriptstyle s_t\ni k} p_t = \pi_k
  \sum_{ t; s_t\ni k} p_t = \pi_k
		\mbox{ for all } k\in U.
		\label{eqq}
	\end{equation}
	%The random sample~$S$ is a random subset such that $Pr(S=s)=p(s).$
	The second-order inclusion probabilities (SIP)  can be represented in matrix form as follows
 $$\Pi = [\pi_{k\ell} ]_{N\times N},\;\;\;\text{where }\pi_{k\ell}=\sum_{ t; s_t\ni k,\ell} p_t,\;\;\; k,\ell\in U.$$
 Also entropy of design $\bm{p}$ is defined as
 $$H(\bm{p})=-\sum_{t=1}^{T}p_t\ln(p_t).$$
 Now consider $y_k, x_k, k\in U$ as the main and auxiliary variables respectively. The primary goal of sampling is typically to estimate parameters such as the population total, denoted by, $$Y=\sum_{k\in U}y_k.$$

 The estimation of parameters including a total and its variance is based on the designs, FIP, and SIP. Given $\pi_k>0$ for all $k\in U$, an unbiased estimate of $Y$ is the Narain-Horvitz-Thompson \citep{nar:51, hor:tho:52} estimator (NHT),
	$$\hat{Y}=\sum\limits_{k\in \mathbb{S}}\frac{y_k}{\pi_k},$$
	with
	\begin{equation}\label{var}
		var(\hat{Y})=\sum\limits_{k\in U}\sum\limits_{\ell\in U}\frac{y_k}{\pi_k}\frac{y_\ell}{\pi_\ell}(\pi_{k\ell}-\pi_k\pi_\ell),
	\end{equation}
	and given that all the $\pi_{k\ell}$ are positive, an alternative unbiased estimator of the variance is
	\begin{equation}\label{varest}
		\hat{var}(\hat{Y})=\sum\limits_{k\in \mathbb{S}}\sum\limits_{\ell\in \mathbb{S}}\frac{y_k}{\pi_k}\frac{y_\ell}{\pi_\ell}\frac{\pi_{k\ell}-\pi_k\pi_\ell}{\pi_{k\ell}}.
	\end{equation}
    
This conventional approach employs designs rooted in theoretical and mathematical algorithms, which, while well-established, may pose challenges in certain scenarios. In particular, imposing additional constraints or requirements within these conventional designs often leads to increased complexity. This complexity may manifest itself in difficulties in implementing constrained designs with a given set of FIP, or in greater challenges in calculating the corresponding SIP, both of which are essential for accurate estimation and assessment. For example, constraints such as requiring all SIP to be positive, enabling or preventing the joint selection of certain units, or selecting units according to a stream (e.g., time order or sequence) can substantially increase the complexity of the design or its implementation.

In finite population sampling, improving the efficiency of sampling designs frequently requires introducing such constraints, which can complicate probability calculations. As an example, consider probability proportional to size sampling (PPS). In this unequal probability design, sample units are selected proportional to the size of an auxiliary variable, which leads to an efficient design when the auxiliary variable is correlated with the main variable of interest. The design with unequal probabilities with replacement was first introduced by \citet{han:hur:43}, while \citet{mad:49}, \citet{nar:51}, and \citet{hor:tho:52} proposed without-replacement versions of PPS. Selecting a PPS sample without replacement and with a fixed sample size can involve intricate procedures, particularly when additional design constraints are introduced. For this purpose, many different designs have been proposed, 50 of which are listed in \citet{bre:han:83} and \citet{til:06,till:2020}. All these designs can implement PPS that satisfies FIP (leading to unbiased estimation of certain parameters using the NHT estimator), but they may result in different SIP, which in turn affects the precision of the NHT estimator.

Although various methods exist for implementing unequal probability, without-replacement, fixed-size designs, identifying a design within this space that achieves higher precision under specific constraints remains a challenging and active area of research. For example, the cube method proposed by \citet{dev:til:04a} provides an elegant and efficient algorithm for balanced sampling, particularly when balancing on auxiliary variables is of primary importance. More recently, determinantal sampling designs, as explored by \citet{loonis2019determinantal} and \citet{loo:23}, offer a powerful and flexible framework, enabling exact control of inclusion probabilities through parameterization of Hermitian contracting matrices. These designs possess desirable correlation structures inherent in determinantal processes, which often lead to favorable variance properties in NHT estimators across multiple auxiliary variables. 
Accordingly, investigating the combinations of SIP that yield an optimal PPS continues to complement these advanced methodologies and remains an area of considerable interest.

Building on these existing methods and their limitations, this paper introduces a new approach aimed at offering additional flexibility and intuitive design capabilities. The core concept of this new approach is to provide a flexible graphical method for implementing equal- and unequal-probability, without-replacement, fixed-size sampling designs. The proposed graphical framework allows for the preservation of given FIP while offering flexibility in adjusting the SIP. While a simplified version of the method corresponds to the systematic sampling approach of \citet{mad:49}, the following sections demonstrate that, starting from a simple configuration (e.g., Madow’s systematic sampling) and iteratively adjusting the SIP under a fixed FIP, the procedure can generate a broad (typically continuous) family of feasible designs with finely tuned joint probabilities. By choosing suitable update rules or objective functions, one can obtain fixed-size designs or steer the construction toward solutions that optimize a chosen criterion. This versatility enables smooth transitions between distinct designs within a single framework and connects finite-population sampling with intelligent algorithmic searches for optimal designs (e.g., lower variance, improved spatial spread, or other application-driven criteria).

 The remainder of the paper is organized as follows.
Section~\ref{Sec.Main} presents the new approach.
Section~\ref{Sec.Reduce} develops procedures for generating new designs by adjusting SIP under fixed FIP.
Section~\ref{Sec.Fix} gives a fixed-size implementation, shows its baseline equivalence to \citet{mad:49}, and demonstrates how the same SIP-adjustment mechanism generates alternative fixed-size designs. Section~\ref{Sec:Inovation} introduces an intelligent method of searching for an optimal design. In Section \ref{Sec:Simulation} some simulations are conducted. Finally, Section~\ref{Sec.Clu} concludes the article with a summary and suggestions.

\section{Graphical Finite Population-Sampling (GFS)}\label{Sec.Main}
For sampling with a predetermined FIP $\bm{\pi}$, in this study a graphical approach is proposed, as presented in Algorithm~\ref{Al00};
	
	\begin{algorithm}[H]
		\caption{GFS} \label{Al00}
		\begin{algorithmic}[1]
			\State Create a two-dimensional coordinate system, indicate the population units on the horizontal axis, and consider the vertical axis for the FIP between 0 and 1,
			\For{$k=1,2,\dots,N$}
			\State Draw a bar of length $\pi_k$ above point $k$ at an arbitrary or random position such that the bar is completely between 0 and 1,
			\EndFor
			\State Select a random point $r$ between 0 and 1, plot it on the vertical axis, and draw a horizontal line, say a random line, from the selected point,
			\State The final sample units are all units for which the random line passes through the bars of their FIP.  
		\end{algorithmic}
	\end{algorithm}
Figure~\ref{fig:3random} provides an example of Algorithm~\ref{Al00} with \begin{equation}\label{FIPexample}
    \bm{\pi} = \{.38, .30, .42, .65, .25, .10, .90\},
\end{equation} each in different color.
\begin{figure}[!htbp]
    \centering
    \includegraphics[width=120mm]{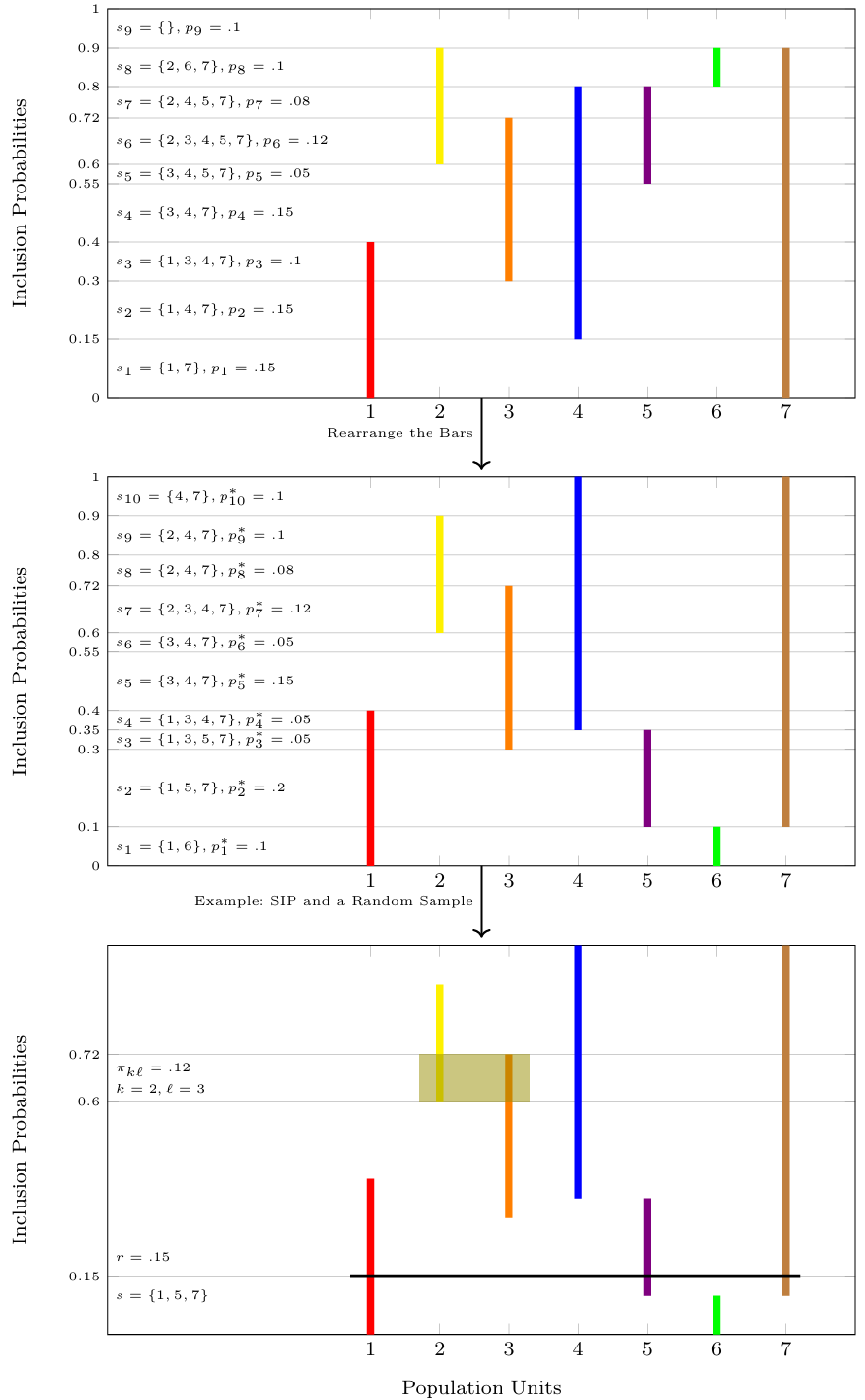}
 \caption{ The top two plots depict different arbitrary arrangements of bars with $\bm{\pi} = \{.38, .30, .42, .65, .25, .10, .90\}$ shown in different colors based on Algorithm \ref{Al00}, in which both arrangements respect the specified FIP. The designs ($\bm{p}$ and $\bm{p^*}$) are calculated and presented to the left of the plots, showing that any change in the position of the bars results in creating a new design. The bottom plot illustrates two examples: (1) The calculation of the SIP for units $k=2$ and $\ell=3$, highlighted in gray, resulting in $\pi_{k\ell}=.12$, (2) A random line (horizontal black line) is drawn, selecting units $1$, $5$ and $7$ as the final sample, $s=\{1,5,7\}$.
}
    \label{fig:3random}
\end{figure}
The top two plots in Figure~\ref{fig:3random} illustrate two different arrangements of bars, each representing a distinct design. Both arrangements adhere to the FIP constraint. The accompanying designs (samples and probabilities) to the left illustrate that any modification in the position of the bars results in a new design, as shown by the transition from design $\bm{p}$ with $T = 9$ in the top plot to design $\bm{p^*}$ with $T = 10$ in the middle plot, where the bars for $\pi_4$, $\pi_5$, $\pi_6$, and $\pi_7$ have been rearranged. It is also evident that in the middle plot, $s_5$ could be merged with $s_6$, and $s_8$ with $s_9$, which would yield $T = 8$ — representing essentially the same design as that depicted with $T = 10$.

The bottom plot in Figure \ref{fig:3random} presents two examples in design $\bm{p^*}$ to illustrate different concepts related to GFS: 
\begin{itemize}
    \item The first example demonstrates how to calculate the SIP for units, which can be determined by finding the intersection of the bars. In this illustration, the second-order inclusion probability for the units $k=2$ and $\ell=3$ is highlighted in gray, resulting in a value of $\pi_{k\ell}=.12$ as specified by the design.
\item The second example involves drawing a random line, shown as a horizontal black line, $r=.15$. This random line passes through the FIP of units $1$, $5$ and $7$, thereby selecting these three units as the final sample. This visual demonstrates how a random selection process can determine the final sample based on GFS.
\end{itemize}

Using GFS, one can obtain an unbiased estimator of \(Y\) while preserving the FIP. Subsequently, by rearranging the bars, the SIP can be modified, allowing control over the estimator’s variance (and hence its precision).
	
 The next Result shows that the method preserves the FIP.
	
	\begin{theorem}
		In the GFS approach, with any arrangement of the bars,
		$$Pr(k\in \mathbb{S})=\pi_k;\;\; k=1,2,...,N,\;\;\;E(n_\mathbb{S})=\sum\limits_{k\in U}\pi_k.$$
	\end{theorem}
 \begin{proof}
     The proof follows directly from Algorithm \ref{Al00}.
 \end{proof}
 Algorithm \ref{Al00} suffers from two key limitations: a high propensity for zero SIP and variable sample size. As an example consider the top plot in Figure \ref{fig:3random}, where the sample size fluctuates among $n_s=0,2,3,4,5$, and the SIP for $k=1, \ell=2,5,6$ are all zero. These deficiencies warrant further exploration, as they significantly impact the algorithm's efficiency and applicability in real-world scenarios. Subsequent sections present novel algorithms specifically developed to overcome these limitations, significantly improving the overall effectiveness of the sampling process.

\section{Generating New Designs by Adjusting SIP}\label{Sec.Reduce}

Even a zero element in the SIP matrix precludes the existence of an unbiased variance estimator for the NHT estimator. Therefore, while a design with a fully positive SIP matrix does not guarantee higher efficiency, it is still of interest to construct such designs, as they allow for unbiased variance estimation. In this section, an algorithm is proposed for the reduction of the zero elements in the SIP matrix.

 To begin, please note that initially each bar represents the complete portion associated with a given $\pi_k$. 
When a bar is considered in smaller parts—either selected as a piece or obtained by dividing it during the algorithm—these parts are referred to as \emph{segments}.
Then, consider the $ y $-axis partitioned into $ D $ strips, denoted by $ \Delta_1, \Delta_2, \dots, \Delta_D $, corresponding to the samples $ s_1, s_2, \dots, s_D $. Note that $ D \geq T $, as some samples may appear multiple times across the $y$-axis partitioning. Now, select two strips, $ \Delta_i $ and $ \Delta_j $, with respective heights $ p_i $ and $ p_j $.
Inside each strip, consider two smaller substrips $\delta_i$ and $\delta_j$, both of size $v_{i,j}(\alpha) = \alpha\times \min(p_i, p_j)$ for $\alpha \in [0,1]$, and define:

\begin{eqnarray*}
    \delta_{i}(k)=\left\{
    \begin{matrix}
        1 & \;\;\;\; \text{if the bar for unit $k$ completely covers the height of substrip $\delta_i$}, \\\\
        0 & \;\;\;\; \text{otherwise}.
    \end{matrix}
    \right.
\end{eqnarray*}

Next, if $s_{i-j} = s_i \setminus s_j$ is nonempty, select $k \in s_{i-j}$ and interchange the segment of $\pi_k$ in $\delta_i$ with $\delta_j$. In other words, set $\delta_{i}(k) = 0$ and $\delta_{j}(k) = 1$, which potentially leads to the following outcomes:

\begin{itemize}
    \item sample $s_i$ with probability $p_i$ splitting into two samples:
    \begin{itemize}
        \item[1:] $s_{i,1} = s_i,\;\;\;\;\;\;\;\;\;\;\;\;\;\; p_{i,1} = p_i - v_{i,j}(\alpha)$,
        \item[2:] $s_{i,2} = s_i \setminus \{k\},\;\;\;\;\; p_{i,2} = v_{i,j}(\alpha)$,
    \end{itemize}
    \item sample $s_j$ with probability $p_j$ splitting into two samples:
    \begin{itemize}
        \item[3:] $s_{j,1} = s_j,\;\;\;\;\;\;\;\;\;\;\;\;\;\; p_{j,1} = p_j - v_{i,j}(\alpha)$,
        \item[4:] $s_{j,2} = s_j \cup \{k\},\;\;\;\;\; p_{j,2} = v_{i,j}(\alpha)$.
    \end{itemize}
\end{itemize}

In summary, to potentially reduce zero SIP and generate new designs, in Algorithm~\ref{Al00} we randomly select a segment of a bar and move it to a new position, provided a vacant location is available. Details are given in Algorithm~\ref{Al3}.

Before introducing the algorithm, it is important to clarify the use of the term “randomly” in this context.
Randomly selecting two strips refers to a random mechanism that can be implemented in different ways. 
In the simulations presented in this paper, we use a simple random sample of size two from all possible pairs of strips. 
Alternatively, the selection can be made with probabilities proportional to size, where the value $p_i$ denotes the height of strip $i$, which in turn represents the design weight of the respective sample. 
The choice of selection method may influence the rate of convergence of the algorithm toward the intended design properties.

\begin{algorithm}
    \caption{Chaotic GFS: Design Generator} \label{Al3}
    \begin{algorithmic}[1]
        \State Create a two-dimensional coordinate system, placing the population units on the horizontal axis and using the vertical axis for the FIP, ranging from 0 to 1, and build the sampling design,
        \For{$m = 1, 2, \dots, M$}
            \State Randomly select two strips $\Delta_i$ and $\Delta_j$, and within them two substrips $\delta_i$ and $\delta_j$ of size $v_{i,j}(\alpha)$ with $0<\alpha<1$, 
            \If{$s_{i-j}$ is nonempty}
            \State Randomly select $k\in s_{i-j}$,
                \State Set $\delta_{i}(k) = 0$ and $\delta_{j}(k) = 1$,
            \EndIf
        \EndFor
        \State Draw a random line,
        \State The final sample consists of all units for which the random line passes through their respective bars of the FIP.
    \end{algorithmic}
\end{algorithm}
To elucidate the core concept of Algorithm~\ref{Al3}, examine the plots in Figure~\ref{fig:2strip}. In the top plot, consider $\Delta_2$ and $\Delta_5$ and within them observe that $\delta_{5}(4)=1$ and $\delta_{2}(4)=0$. Subsequently, we can truncate the corresponding segment of $\pi_4$ in $\delta_5$ and relocate it to the new position in $\delta_2$, resulting in $\delta_{2}(4)=1$ and $\delta_{5}(4)=0$, as depicted in the bottom plot. This implementation of Algorithm~\ref{Al3} with $M=1$ generates a new design with more possible samples. For example, in the top plot, $\pi_{k\ell}=0$ for $k=4$ and $\ell=5$, while in the new design, we have $\pi_{k\ell}=0.07$ for $k=4$ and $\ell=5$.

It is noteworthy that, in GFS, the SIP are exactly computable. For any refinement level $M$, Algorithm~\ref{Al3} proceeds under fixed FIP until the configuration stabilizes; sampling is then carried out from this final design, and the operative SIP are those induced by that stabilized configuration.
\begin{figure}[!htbp]
    \centering
    \includegraphics[width=120mm]{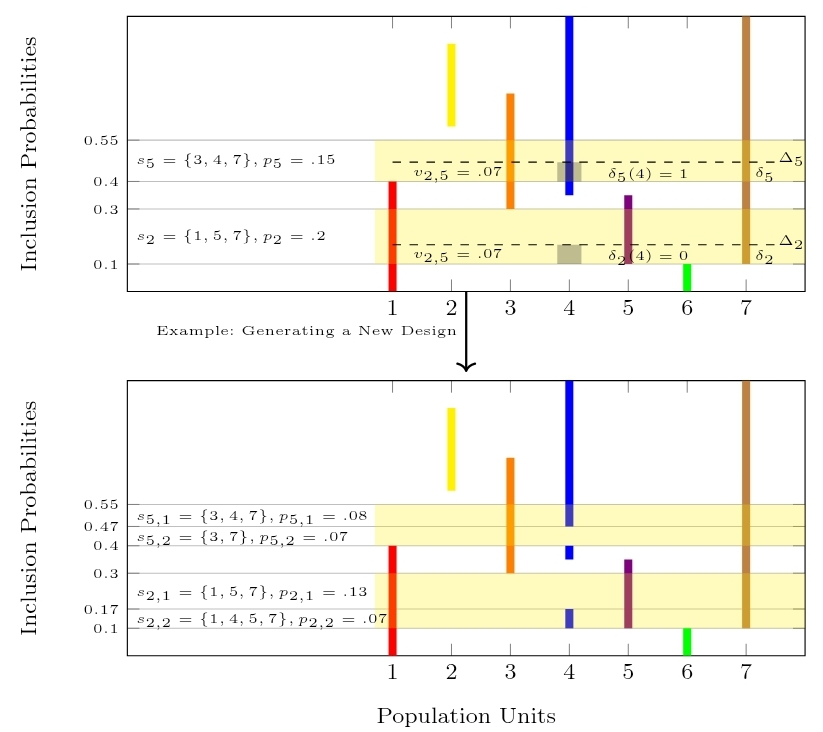}
 \caption{
 An example of generating a new design $\bm{p^*}$ (the middle plot of Figure~\ref{fig:3random}). By selecting $\Delta_2$ and $\Delta_5$ with heights $p^*_2 = 0.2$ and $p^*_5 = 0.15$, and choosing substrips $\delta_2$ and $\delta_5$ of size $v_{2,5}(7/15)=0.07$, it is possible to interchange segments between substrips $\delta_2$ and $\delta_5$ for unit $k=4$, which yields four samples: $s_{2,1}$, $s_{2,2}$, $s_{5,1}$, and $s_{5,2}$.
}
    \label{fig:2strip}
\end{figure}

In Algorithm~\ref{Al3}, each FIP-preserving segment exchange empirically disperses joint mass over a larger set of feasible samples and tends to reduce the number of zero entries in the SIP matrix. Although monotonic improvement at every step is not claimed, this behavior suggests a progressive increase in distributional spread, motivating the following conjecture for future work on this class of algorithms.

\begin{conjecture}\label{Conj.Poiss}
Chaotic GFS converges to Poisson sampling as $M \to \infty$.
\end{conjecture}

	\section{Fixed-size Algorithm of GFS}\label{Sec.Fix}
 Building upon the resolution of the zero SIP issue in the previous section, this section addresses the challenge of random sample size within the framework of Algorithm \ref{Al00}.
 
	To overcome this challenge, one of the simplest approaches is to arrange the bars based in a cumulative manner as follows. The first bar starts from zero to $\pi_1$, the second bar starts from $\pi_1$ to $\min{(\pi_1+\pi_2, 1)},$ and if the minimum is 1, then the remainder of $(\pi_1+\pi_2)-1$ wraps around from zero, and so forth. Details are presented in Algorithm \ref{Alfixed}. 
	\begin{algorithm}[H]
		\caption{Fixed-size GFS, equivalent to \citet{mad:49}} \label{Alfixed}
		\begin{algorithmic}[1]
			\State In Algorithm \ref{Al00}, the first bar covers $[0,\pi_1]$, and set $b_1=\pi_1$,
			\For{$k=2,3,...,N$}
			\If{$\pi_k+b_{k-1}\leq 1$} $k^{th}$ bar covers $(b_{k-1},\pi_k+b_{k-1}]$, set $b_k=\pi_k+b_{k-1}$,
			\Else{$\;\;\;\;\;\;\;\;\;\;\;\;\;\;\;\;\;\;\;\;\;\;\;\;\;\;\;\;\;$} $k^{th}$ bar covers $(b_{k-1},1]\cup [0,\pi_k+b_{k-1}-1]$, set $b_k=\pi_k+b_{k-1}-1$,
            \EndIf
			\EndFor
		\end{algorithmic}
	\end{algorithm}

\begin{figure}[h]
    \centering
    \includegraphics[width=120mm]{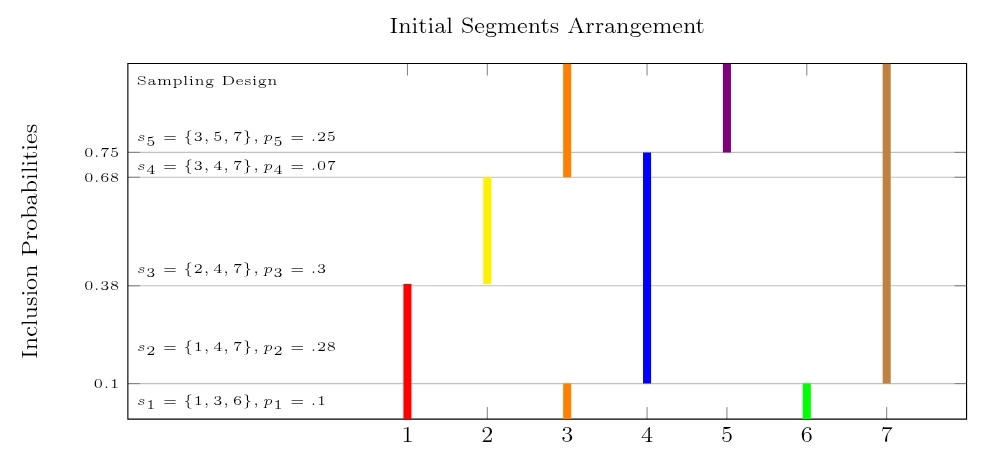}
\caption{The plot depicts a fixed-size version (Algorithm \ref{Alfixed}) of GFS with $\bm{\pi} = \{.38, .3, .42, .65, .25, .1, .9\}$, arranged sequentially to create a fixed-size design.}
    \label{fig:fixed}
\end{figure}
Figure \ref{fig:fixed} presents a fixed-size version of GFS with $\bm{\pi} = \{.38, .3, .42, .65, .25, .1, .9\}$ arranged sequentially to form a design with fixed size $n=\sum\limits_{k=1}^7\pi_k=3$. In this arrangement, the FIP are aligned to maintains the fixed size constraint. 

Interestingly, Algorithm \ref{Alfixed} can be viewed as a vertical adaptation of the systematic sampling method proposed by \citet{mad:49}. However, in the following, it will be demonstrated that unlike the systematic sampling described by \citet{mad:49}, this novel approach possesses the capability to emulate various traditional and novel designs, including many fixed-size unequal probability sampling methods.

To generate a new design from the Madow version of GFS while preserving the FIP and the fixed-size constraint, a strategy analogous to that described in Section~\ref{Sec.Reduce} can be employed.
 Specifically, two strips corresponding to two samples can be selected, and two segments of their bars can be interchanged, provided that an empty space exists in the new positions. This approach preserves fixed sample sizes while generating a new design. However, to present a complete algorithm for this process, it is necessary to define two key concepts.

\begin{definition}
    Consider two substrips on the vertical axis, $\delta_i\subseteq \Delta_i$ and $\delta_j\subseteq \Delta_j$. Two segments, $\delta_{i}(k)$ and $\delta_{j}(\ell)$, are interchangeable if 
    $$\delta_{i}(k) = 1, \;\;\; \delta_{j}(\ell) = 1;\;\;\;\;\; \delta_{j}(k) = 0, \;\;\; \delta_{i}(\ell) = 0.$$
\end{definition}

\begin{definition}
    Interchanging two interchangeable segments $\delta_{i}(k)$ and $\delta_{j}(\ell)$ means setting 
    $$\delta_{i}(k) = 0, \;\;\; \delta_{j}(\ell) = 0;\;\;\;\;\; \delta_{j}(k) = 1, \;\;\; \delta_{i}(\ell) = 1.$$
\end{definition}

Now, consider two strips $\Delta_i$ and $\Delta_j$ with heights $p_i$ and $p_j$, and two substrips $\delta_i$ and $\delta_j$ of the same size $v_{i,j}(\alpha)$. If $s_{i-j}$ and $s_{j-i}$ are nonempty, select $k \in s_{i-j}$ and $\ell \in s_{j-i}$ randomly, and interchange them, potentially leading to:

\begin{itemize}
    \item Sample $s_i$ with probability $p_i$ splits into two samples:
    \begin{itemize}
        \item[1:] $s_{i,1} = s_i,\;\;\;\;\;\;\;\;\;\;\;\;\;\;\;\;\;\;\;\; p_{i,1} = p_i - v_{i,j}(\alpha)$,
        \item[2:] $s_{i,2} = s_i \cup \{\ell\} \setminus \{k\},\;\; p_{i,2} = v_{i,j}(\alpha)$,
    \end{itemize}
    \item Sample $s_j$ with probability $p_j$ splits into two samples:
    \begin{itemize}
        \item[3:] $s_{j,1} = s_j,\;\;\;\;\;\;\;\;\;\;\;\;\;\;\;\;\;\;\;\; p_{j,1} = p_j - v_{i,j}(\alpha)$,
        \item[4:] $s_{j,2} = s_j \cup \{k\} \setminus \{\ell\},\;\; p_{j,2} = v_{i,j}(\alpha)$.
    \end{itemize}
\end{itemize}
Details are presented in Algorithm \ref{Al2}.
\begin{algorithm}[H]
    \caption{Chaotic fixed-size GFS} \label{Al2}
    \begin{algorithmic}[1]
        \State Implement Algorithm \ref{Alfixed}.
        \For{$m = 1, 2, \dots, M$}
            \State Select randomly two strips $\Delta_i$ and $\Delta_j$, and within them two substrips $\delta_i$ and $\delta_j$ with size $v_{i,j}(\alpha)$, $0<\alpha<1$, 
            \If{$s_{i-j}$ and $s_{j-i}$ are nonempty},
            \State select $k\in s_{i-j}$ and $\ell\in s_{j-i}$ randomly,
                \State Interchange $\delta_{i}(k)$ and $\delta_{j}(\ell)$.
            \EndIf
        \EndFor
    \end{algorithmic}
\end{algorithm}
 
 As an example of Algorithm \ref{Al2}, consider Figure \ref{fig:2fixed}, in which two strips, $ \Delta_2 $ and $ \Delta_5 $, are selected and by using two substrips of size $ v_{2,5}(10/25) = .1 $, segments of units 1 and 5 are interchanged to generate a new design.

\begin{figure}[h]
    \centering
    \includegraphics[width=120mm]{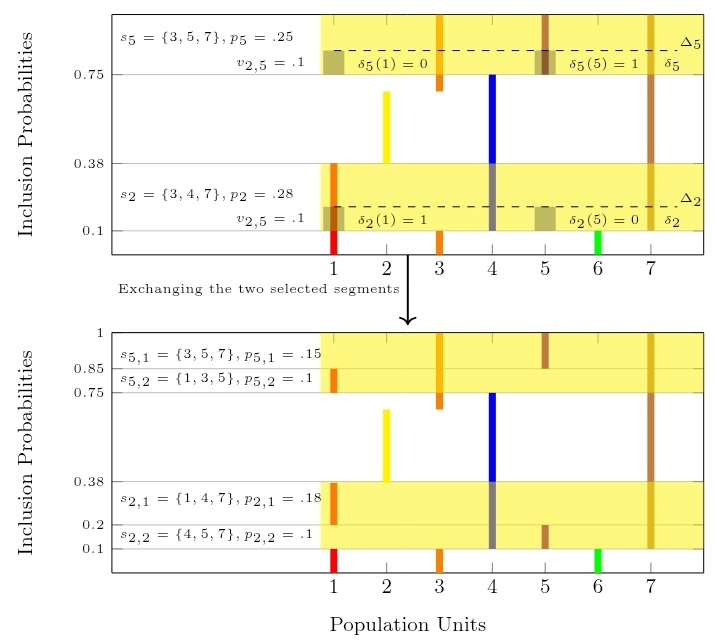}
\caption{The top plot depicts a fixed-size version of GFS with $\bm{\pi} = \{0.38, 0.30, 0.42, 0.65, 0.25, 0.10, 0.90\}$, arranged sequentially to create a fixed-size design in which two strips, $\Delta_2$ and $\Delta_5$, are selected. Within these strips, two substrips of size $v_{2,5} = v_{2,5}(10/25) = 0.10$ are indicated by dashed lines. Since units $1$ and $5$ are interchangeable in this situation, they are interchanged, resulting in two new samples, $s_{2,2}$ and $s_{5,2}$, alongside the original samples, $s_{2,1}$ and $s_{5,1}$, depicted in the bottom plot.}

    \label{fig:2fixed}
\end{figure}

The following theorem states the basic invariants of the chaotic fixed-size procedure for Algorithm~\ref{Al2}.

\begin{theorem}
		In the chaotic fixed-size GFS approach, based on 
  Algorithms \ref{Al2},
		$$Pr(k\in \mathbb{S})=\pi_k;\;\; k=1,2,\dots,N,$$
$$E(n_\mathbb{S})=\sum\limits_{k\in U}\pi_k,$$
  and if $E(n_\mathbb{S})$ is an integer, the design is fixed-size; equivalently:
  $$var(n_\mathbb{S})=0.$$
	\end{theorem}
 \begin{proof}
     The proof follows directly from Algorithm \ref{Al2}.
 \end{proof}
 In parallel with Conjecture~\ref{Conj.Poiss}, with nothing that the chaotic fixed-size procedure preserves the prescribed FIP and the sample size while empirically spreading probability mass across feasible size-$n$ samples; in particular, it tends to increase entropy subject to $|\mathbb{S}|=n$. This motivates the following conjecture that the limiting design coincides with the conditional Poisson design—the fixed-size analogue often described as maximum-entropy sampling \citep{til:06}.
 \begin{conjecture}
     Chaotic fixed-size GFS with $M \rightarrow \infty$, converges to  maximum-entropy sampling.
 \end{conjecture}

Figure~\ref{Fig.MaxEnt} illustrates a fixed-size design for the population~\eqref{FIPexample} generated by Algorithm~\ref{Al2} with $\alpha=0.5$ and $M=10{,}000$, characterized by broadly distributed design probabilities and substantially fewer zero SIP entries.

 \begin{figure}[H]
		\centering
		\includegraphics[width=75mm]{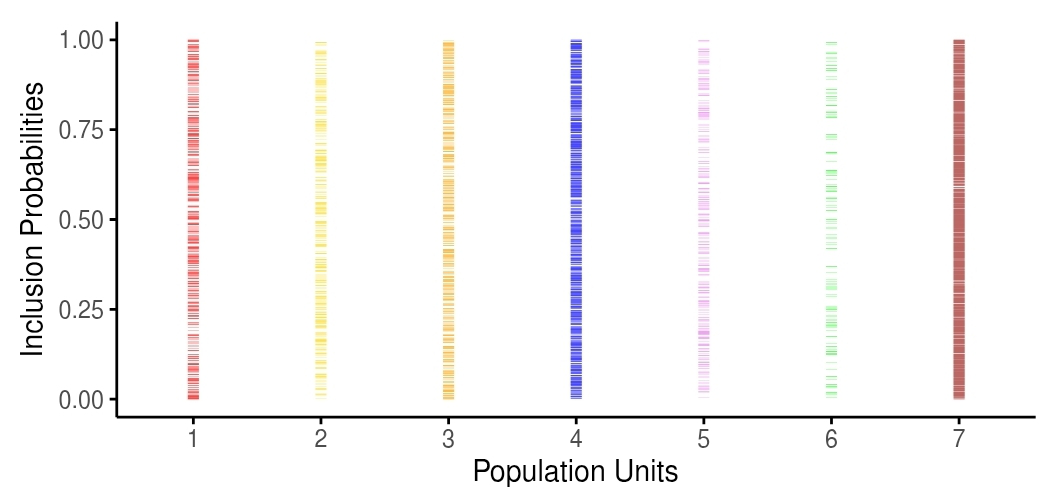} 
		\caption{Implementing of Algorithm \ref{Al2} on the data of population \eqref{FIPexample} with $\alpha = .5$ and $M = 10,000$.}\label{Fig.MaxEnt}
	\end{figure}

The framework preserves the prescribed FIP, supports fixed-size sampling, and permits the SIP to be tuned through local updates, thereby generating a broad family of feasible designs with exactly computable SIP. These properties make it natural to embed the construction within an optimization loop that seeks designs aligned with problem-specific goals while respecting the prescribed constraints.

 \section{Optimal GFS (OGFS)}\label{Sec:Inovation}
As noted by \citet{thompson1996adaptive}, there is no globally optimal sampling design for finite population sampling. The optimality of a design depends on several factors, including the structure of the main variable, the FIP, and the estimator used. In discussing optimal design, the focus is on how, given a vector of FIP, it is possible to implement an unequal probability sampling method that outperforms conventional well-known sampling designs. It is also worth noting that this discussion primarily aims to enhance the design stage when using the NHT estimator. Further research could explore improving both the design and estimation stages using GFS, which presents an interesting avenue for future investigation.

With OGFS, the GFS's ability to generate diverse designs illustrates a novel pathway, potentially offering a simple and efficient means for discovering new, effective designs.

In chaotic (random- or fixed-size) GFS, the FIP remain fixed, which facilitates the creation of designs with a very large and flexible set of possible SIP by rearranging the bars. Leveraging proper auxiliary variables, it becomes feasible to construct numerous designs. Once the position of the bars (segments) is fixed, the design becomes fully known in GFS, enabling the calculation of criteria such as the variance, maximum error, etc., of the NHT estimator based on some auxiliary variables. Therefore, one approach is to fix a criterion, rearrange the bars to change the SIP, generate candidate designs, and select the best according to that criterion.
In the simplest case, consider two auxiliary variables $x$ and $z$ with some reasonable correlations with $y$, one for building the FIP and another for evaluating designs based on some criteria. Using
	\begin{equation}\label{PPSP}		\pi_k=min(cx_k,1),\;\;\;\;\text{where}\;\;\;\; \sum\limits_{k\in U}\pi_k=n,
	\end{equation}
leads to FIP based on $x$ with sample size $n$, known as PPS. In this design, we know that $var(\hat{X})=0$. Then using an evaluation variable $z$, consider a parameter, say $\theta_z$, rearranging the bars a large number of times based on a random or intelligent algorithm, and selecting the best arrangement to estimate $\theta_z$ based on the NHT estimator of $z$, results in selecting an optimal design. 

Any search algorithm over designs requires a criterion to assess its proximity to the desired outcome. In the context of optimizing GFS, various design parameters can be considered as a criterion. The following equations outline three such criteria for OGPS:

\begin{eqnarray*}\label{crit}
C_1(\theta_z=Z,\bm{p}) &=& \sum\limits_{t=1}^T(\hat{Z}_{s_t}-Z)^2 p_t,\nonumber\\
C_2(\theta_z=Z,\bm{p}) &=& \sum\limits_{t=1}^T |\hat{Z}_{s_t}-Z| p_t,\nonumber\\
C_3(\theta_z=Z,\bm{p}) &=& \max_{t=1,\dots,T} |\hat{Z}_{s_t}-Z|,
\end{eqnarray*}

where 

\[
Z = \sum\limits_{k=1}^N z_k, \quad \hat{Z}_{s_t} = \sum\limits_{k \in s_t} \frac{z_k}{\pi_k}.
\]

It is worth noting that while the process of optimizing GFS using an auxiliary variable $z$ as a balancing target shares conceptual similarities with balanced sampling methods \citep{dev:til:04a, loo:23}, OGFS does not explicitly enforce balance constraints at each step. Instead, it seeks designs that minimize a specified criterion (such as variance or maximum error) through intelligent rearrangement of segments, which may lead to balanced or near-balanced samples as a byproduct rather than as a direct requirement.

\subsection*{An Intelligent Design Search Algorithm}
\label{sec:algorithm_explanation}

A variation of a greedy best-first search approach is presented to search for an optimal design. It begins with the initialization of auxiliary variables: $x$ for constructing FIP and $z$ for evaluating designs. An initial design is then established based on the FIP, serving as the starting point for the optimization process.

The algorithm proceeds through several iterations, say $\Lambda$, where in each cycle, $N_{node}$ number of new candidate designs are generated by modifying the last design based on Algorithm \ref{Al2}. This modification ensures the exploration of diverse configurations. Each newly generated design undergoes evaluation based on a predefined criterion, denoted as $C(\theta_z=Z,\bm{p})$. At the beginning, the initial design is considered the best design. Whenever a superior design is identified, it replaces the previous best design. 

Two critical sets are established during this process: the open set and the closed set. The open set maintains designs, prioritized by their costs, that are yet to be explored. In contrast, the closed set keeps track of designs that have already been evaluated, ensuring that the algorithm does not redundantly process the same design multiple times. Throughout the iterations, the algorithm tracks the most effective design. Afterwards, the set of potential designs is managed to retain only the most promising candidates, replacing less efficient designs when necessary. This management ensures a focus on the most viable options for further refinement.

After all iterations are completed, the algorithm returns the best design found, representing the optimal configuration according to the defined criterion. For further details, refer to Algorithm \ref{alg:optimize_design}.

\begin{algorithm}[H]
\caption{OGFS: a Greedy Best-First Search Variant}\label{alg:optimize_design}
\begin{algorithmic}[1]
\State \textbf{Inputs:} auxiliary variable $x$ (to construct the FIP), evaluation variable $z$, criterion $C(\theta_z=Z,\bm{p})$ to be \emph{minimized}, iteration budget $\Lambda$, branching factor $N_{\text{node}}$, and capacity $\texttt{max\_open\_set\_size}$.
\State Create the \texttt{initial\_design} from the FIP using Algorithm~\ref{Alfixed}.
\State Initialize $\texttt{open\_set}\gets\{\}$ (priority queue keyed by criterion) and $\texttt{closed\_set}\gets\{\}$ (evaluated designs).
\State Compute $\texttt{init\_criterion}\gets C(\theta_z=Z,\bm{p}(\texttt{initial\_design}))$.
\State Insert $(\texttt{init\_criterion},\,\texttt{initial\_design})$ into \texttt{open\_set}.
\State $\texttt{best\_design}\gets \texttt{initial\_design}$; \quad $\texttt{best\_criterion}\gets \texttt{init\_criterion}$.
\For{$\texttt{iterations}=1$ \textbf{to} $\Lambda$}
  \State Extract and remove $(\texttt{current\_criterion},\,\texttt{current\_design})$ with minimum criterion from \texttt{open\_set}.
  \For{$i=1$ \textbf{to} $N_{\text{node}}$}
    \State Generate $\texttt{new\_design}\leftarrow \textsc{Apply-Algorithm \ref{Al2} }(\texttt{current\_design})$ 
    \State $\texttt{new\_criterion}\leftarrow C(\theta_z=Z,\bm{p}(\texttt{new\_design}))$
    \If{$\texttt{new\_design}\notin \texttt{closed\_set}$}
      \If{$\texttt{new\_criterion}<\texttt{best\_criterion}$}
        \State $\texttt{best\_design}\gets \texttt{new\_design}$; \quad $\texttt{best\_criterion}\gets \texttt{new\_criterion}$
      \EndIf
      \If{$|\texttt{open\_set}|<\texttt{max\_open\_set\_size}$}
        \State Insert $(\texttt{new\_criterion},\,\texttt{new\_design})$ into \texttt{open\_set}
      \Else
        \State Let $(\texttt{worst\_crit},\,\texttt{worst\_des})$ be the element with maximum criterion in \texttt{open\_set}
        \If{$\texttt{new\_criterion}<\texttt{worst\_crit}$}
          \State Replace $(\texttt{worst\_crit},\,\texttt{worst\_des})$ by $(\texttt{new\_criterion},\,\texttt{new\_design})$ in \texttt{open\_set}
        \EndIf
      \EndIf
    \EndIf
  \EndFor
  \State Add $\texttt{current\_design}$ to \texttt{closed\_set}
\EndFor
\State \Return $\texttt{best\_design}$
\end{algorithmic}
\end{algorithm}

 \section{Simulations}\label{Sec:Simulation}
To evaluate Algorithm \ref{alg:optimize_design} in finding optimal designs, a series of simulation studies on both synthetic and real data were conducted. The focus of the study was to compare the efficiency of OGFS relative to the following designs:
\begin{itemize}
    \item \text{SRS}: Simple Random Sampling,
    \item \text{CUB}: Cube Method \citep{dev:til:04a},
    \item \text{DSD}: Determinantal Sampling Design \citep{loo:23}.
\end{itemize}
OGFS was optimized using a variance criterion:
\[
C(\theta_z = Z, \bm{p}) = \sum\limits_{t=1}^T (\hat{Z}_{s_t} - Z)^2 p_t,
\]
where $\hat{Z}_{s_t}$ is the NHT estimator for $Z$ based on sample $s_t$. 

The efficiency for estimating $Y$ was measured as
\[
EF_{SRS} = \frac{\operatorname{var}(\hat{Y}_{\text{SRS}})}{\operatorname{var}(\hat{Y}_{\text{OGFS}})}, \quad 
EF_{DSD} = \frac{\operatorname{var}(\hat{Y}_{\text{DSD}})}{\operatorname{var}(\hat{Y}_{\text{OGFS}})}, \quad 
EF_{CUB} = \frac{\operatorname{\widehat{var}}(\hat{Y}_{\text{CUB}})}{\operatorname{var}(\hat{Y}_{\text{OGFS}})},
\]
with analogous expressions used for the efficiency of estimating $Z$. Here, $\hat{Y}_\cdot$ denotes the NHT estimator based on the corresponding method. 
For SRS, DSD, and OGFS, the variance of the NHT estimator can be computed exactly from the known first- and second-order inclusion probabilities, without the need for Monte Carlo estimation. 
The reported values for these methods are therefore exact. 
For the Cube method, however, the variance was estimated via Monte Carlo simulation, as the exact computation was not straightforward.

For this simulation, a specialized Python package was utilized, \texttt{graphical-sampling} \citep{graphical_sampling_2024}, which implements the Graphical Sampling Approach. This package served as the core framework for the
simulations in this section, enabling efficient and accurate implementation of the sampling techniques discussed. 
All the designs in this simulation section, except for SRS, used the same auxiliary variable both for defining unequal probabilities and for balancing or optimizing the sampling design.

Also, in all simulations, the parameters of Algorithm~\ref{Al2} were set to $M \sim \operatorname{DU}(1,3)$ and $\alpha \sim \operatorname{CU}(0.7, 0.9)$, where DU and CU denote discrete uniform and continuous uniform distributions, respectively.

\subsection{Simulated Population}
For the simulated population, data were generated as follows:
\begin{itemize}
    \item Population size: $N = 100$; sample size: $n = 10$.
    \item The main variable $y$ was generated as $y \sim \mathcal{N}(100, 10^2)$.
    \item Auxiliary variables were generated as
    \[
    z = b_1 y + \epsilon_z, \quad x = b_2 y + \epsilon_x, \quad b_1, b_2 \in \mathbb{R}, \quad \epsilon_z \sim \mathcal{N}(0, \sigma_z^2), \quad \epsilon_x \sim \mathcal{N}(0, \sigma_x^2),
    \]
    where $\sigma_z$ and $\sigma_x$ were selected to yield approximate target correlations of $\rho_{z,y}, \rho_{x,y} \in \{0.75, 0.85, 0.95\}$.
    \item For $x$, an constant-valued vector was also considered in one scenario to implement equal-probability sampling when evaluating OGFS.
\end{itemize}

For each combination of $\rho_{z,y}$ and $\rho_{x,y}$, Algorithm~\ref{alg:optimize_design} was run with $\Lambda = 1000$ iterations and $N_{\text{node}}=30$. The results are shown in Figure~\ref{fig:simupop}. The efficiency of OGFS relative to DSD and CUB is shown as bar charts, while its efficiency relative to SRS is indicated as values annotated above the bars. This presentation was chosen because the high efficiency values relative to SRS would have resulted in disproportionately tall bars, making it difficult to visualize the comparative performance of DSD and CUB.

Efficiency was evaluated across different combinations of correlations between the auxiliary variables $x$ and $z$ and the main variable $y$, where $\rho_{x,y} = 0.00$ corresponds to the case of equal probability sampling. Although the main purpose of sampling is the estimation of the main variable parameter, $Y$, in this study the estimation of $Z$—which serves as the auxiliary variable used to optimize the design in OGFS and to impose balance constraints in DSD and CUB—is also computed and plotted. This allows us to evaluate the performance of each design with respect to the variable that provides the information for optimization and balancing. Then, in the analysis, the main focus is on estimating $Y$.
 
 In general, OGFS demonstrated good performance when $\rho_{x,y} = 0$, representing the equal probability sampling case. In these scenarios, the correlation between $z$ and $y$ plays a crucial role, as the efficiency gain in estimating $Y$ depends on how effectively information from $z$ can support the estimation of $Y$. When moving to unequal probability sampling, the situation changes: the variation of the NHT estimator becomes primarily influenced by linear combinations of $z_k / \pi_k$ for $k \in \mathbb{S}$, making the correlation between $z_k / \pi_k$ and $y_k / \pi_k$ particularly important. As demonstrated, Algorithm~\ref{alg:optimize_design} was able to address these challenges effectively, with OGFS consistently outperforming or at least matching the performance of its competitors, DSD and CUB. The efficiency of OGFS relative to SRS was substantial across all scenarios, and notably, DSD outperformed CUB in almost all cases.

Considering the magnitude of efficiency gains, OGFS exhibited superior performance in estimating the parameter of the auxiliary variable used for optimization, $z$ (top plot). Furthermore, in terms of the number of scenarios where OGFS strictly outperformed its competitors, DSD and CUB, its performance was particularly noteworthy for estimating the parameter of the main variable, $y$ (bottom plot).

Overall, these simulation studies confirm the robustness and efficiency of OGFS across all tested settings, involving various combinations of $\rho_{z,y}$ and $\rho_{x,y}$.

\begin{figure}[h]
    \centering
    \includegraphics[width=130mm]{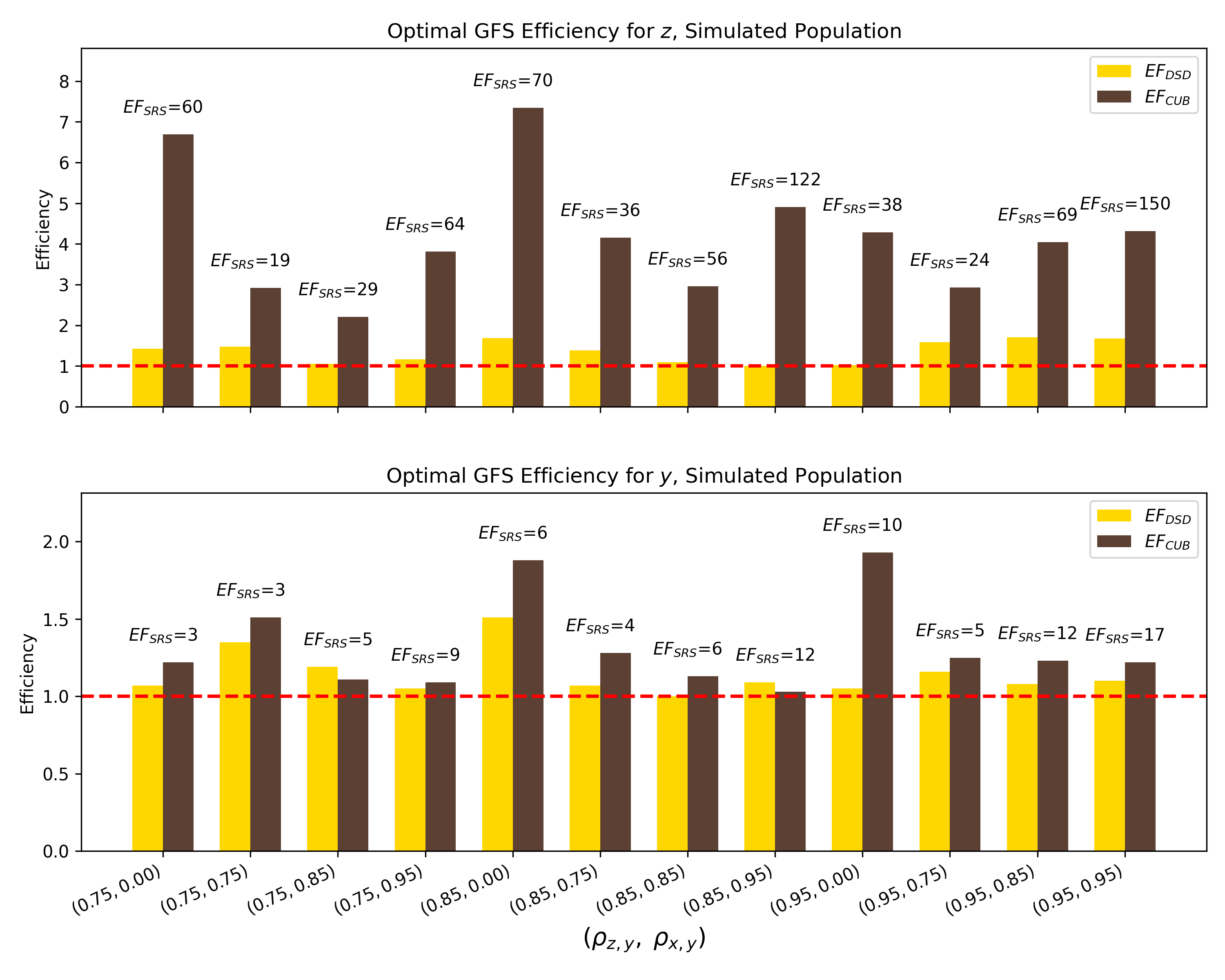}
    \caption{Efficiency of OGFS relative to SRS, DSD, and CUB across different correlation settings in simulated populations. The bars represent efficiency relative to DSD and CUB, while efficiency relative to SRS is annotated as values above the bars to enhance readability, given the large values observed for SRS. The red horizontal line indicates the baseline where efficiency equals 1. Top: efficiency for $z$; Bottom: efficiency for $y$.}
    \label{fig:simupop}
\end{figure}

\subsection{MU284 Data}
Now that Algorithm~\ref{alg:optimize_design} has been evaluated across different simulated scenarios involving correlations between the main and auxiliary variables, this subsection assesses its performance in a real case study, the \texttt{MU284} dataset from \citet{sar:swe:wre:92}, which contains data from 284 Swedish municipalities. To improve the clarity of our visualizations, we removed three outliers. These outliers were not high-leverage points and did not materially affect the results, as correlations remained consistent before and after their exclusion. This assessment considered the following variables:
\begin{itemize}
    \item $y=$ \texttt{RMT85}: revenues from 1985 municipal taxation (in millions of kronor),
    \item $z=$ \texttt{CS82}: number of Conservative seats in the municipal council,
    \item $x=$ \texttt{SS82}: number of Social-Democratic seats in the municipal council.

\end{itemize}
Figure~\ref{fig:mu284} presents the distributions and correlations of the selected variables. The main and auxiliary variables were chosen to ensure that the estimation of $y$ based on $x$ and $z$ was meaningful given the temporal availability of the data. Both $x$ and $z$ exhibited moderate correlations with $y$ ($\rho_{z,y} = 0.46$, $\rho_{x,y} = 0.45$), providing a realistic context to evaluate the effectiveness of Algorithm~\ref{alg:optimize_design}.
\begin{figure}[h]
    \centering
    \includegraphics[width=90mm]{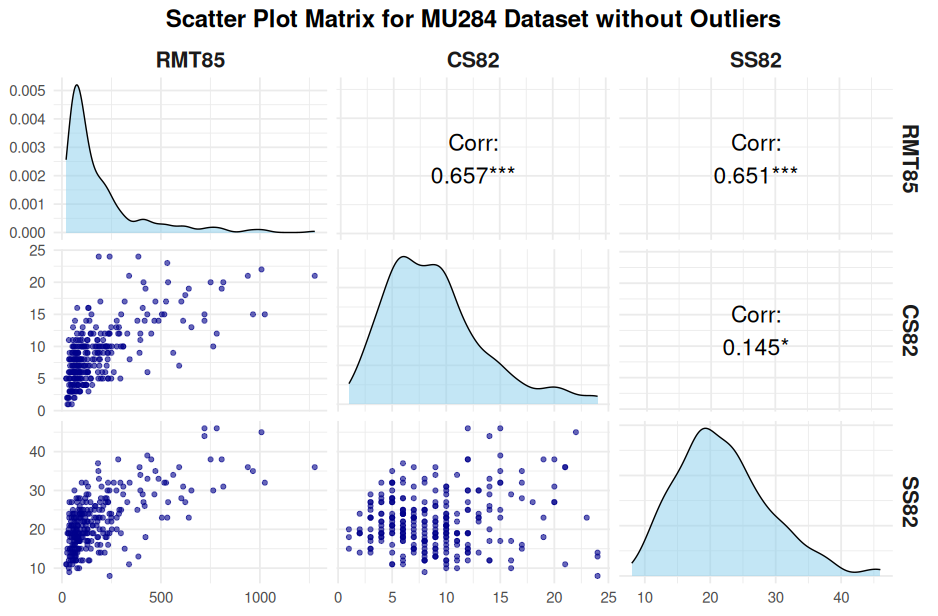}
    \caption{Scatter plot matrix and distributions for the three selected variables from the MU284 dataset after outlier removal. The figure displays Pearson correlation coefficients, with significance levels indicated by symbols: * for $p$-value $< 0.05$ and *** for $p$-value $< 0.001$.}
    \label{fig:mu284}
\end{figure}

In addition to examining the effect of moderate correlation in a realistic setting, the role of $N_{\text{node}}$, an important parameter of Algorithm~\ref{alg:optimize_design}, was investigated. In general, aside from the inherent randomness in the search path, larger values of $N_{\text{node}}$ increase the breadth fo the search per iteration and are expected to lead to greater efficiency of OGFS. To this end, combinations of three sample sizes ($n = 5, 10, 15$) and three levels of $N_{\text{node}}$ ($10$, $30$, and $50$) were considered. The number of iterations was fixed at $\Lambda = 2500$. 

The results are presented in Figure~\ref{fig:realpop}. As before, efficiencies relative to DSD and CUB are shown as bar charts, while efficiencies relative to SRS are annotated as values above the bars.
 The results indicate that, as expected, increasing $N_{\text{node}}$ generally improves the efficiency of OGFS. With sufficiently large values of $N_{\text{node}}$, OGFS outperformed SRS, DSD, and CUB in almost all cases, despite the moderate correlations between $x$ and $y$ and between $z$ and $y$. Here, as in the simulated population, considering the cases in which OGFS strictly outperformed its competitors, the more realistic scenario of estimating the parameter of the main variable $Y$ demonstrated stronger performance compared to the case of estimating the parameter of the auxiliary variable $Z$, particularly when $N_{\text{node}}$ was sufficiently large. However, in terms of the magnitude of the efficiency gains, OGFS exhibited better performance in estimating $z$. This suggests that OGFS is able to leverage information within $z$ more effectively during optimization, thereby improving both the estimation of $z$ and its paired main variable $y$ more efficiently than its rivals, DSD and CUB.

The sample size $n$ showed an generally positive effect on the efficiency of OGFS in estimating $Z$, while for the main variable $y$, this effect was not observed as consistently. Nevertheless, across almost all sample sizes, OGFS outperformed the other designs overall.

An interesting observation is that in many cases, when the algorithm was stopped, the efficiencies of OGFS were still increasing. This suggests that with a higher-performance computing configuration and sufficient compute time, it would be possible to identify even more efficient OGFS designs.

\begin{figure}[h]
    \centering
    \includegraphics[width=130mm]{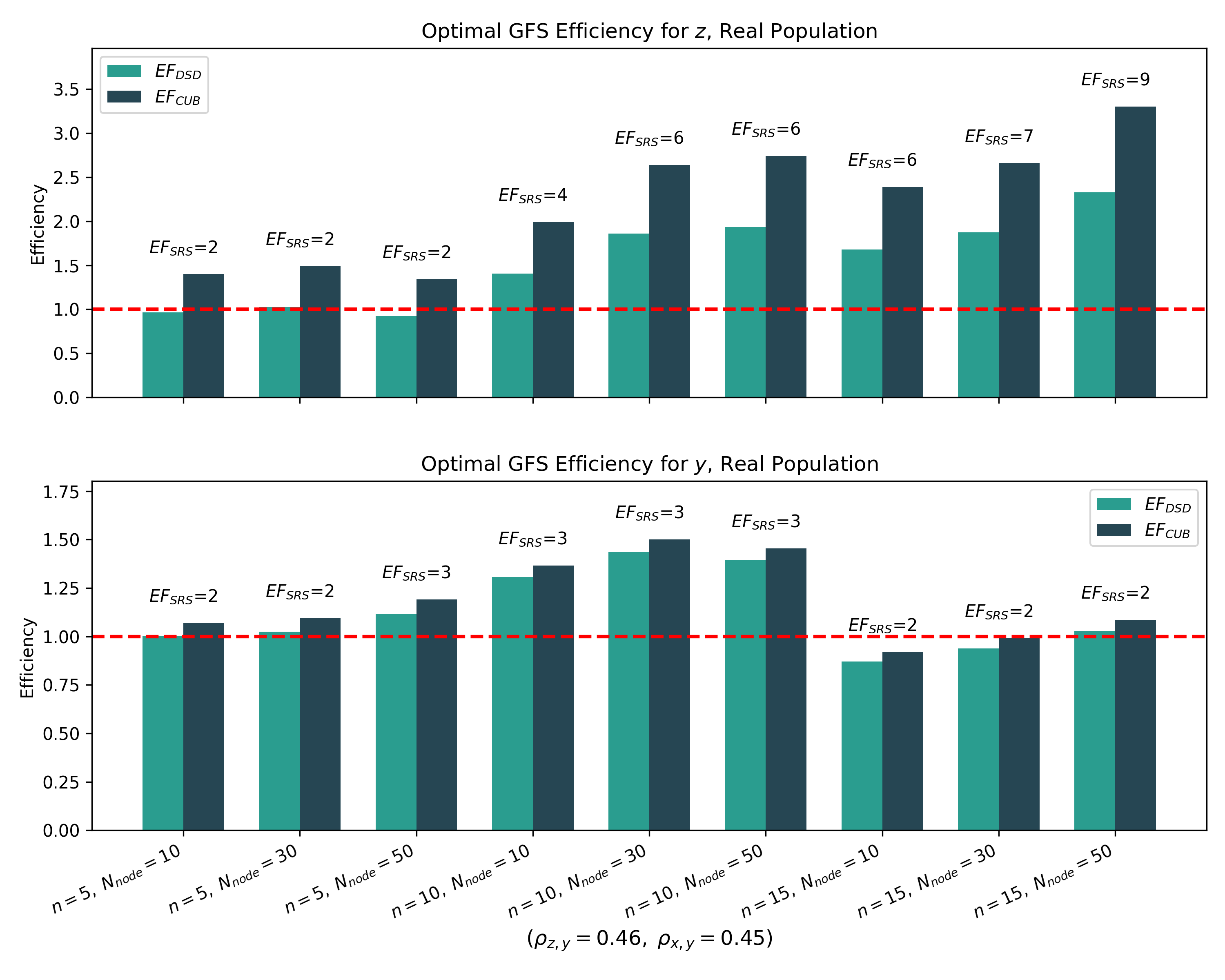}
    \caption{Efficiency of OGFS relative to SRS, DSD, and CUB for the MU284 dataset across different values of $N_{\text{node}}$ and sample sizes. The bars represent efficiency relative to DSD and CUB, while efficiency relative to SRS is annotated as values above the bars to enhance readability, given the large values observed for SRS. The red horizontal line indicates the baseline where efficiency equals 1. Top: efficiency for $z$; Bottom: efficiency for $y$.}
    \label{fig:realpop}
\end{figure}

While there are many ways to further improve Algorithm~\ref{alg:optimize_design}, my aim here is not to provide an exhaustive treatment. Some potential enhancements include:
\begin{itemize}
    \item \textbf{Multi-start strategies:} initiating the search from multiple diverse initial configurations to reduce sensitivity to starting conditions and improve exploration of the design space.
    \item \textbf{Adaptive open set pruning:} dynamically pruning the open set to focus on the most promising candidate designs while maintaining diversity.
    \item \textbf{Shocks/random injection:} introducing new initial designs mid-optimization if the process appears to be trapped in a local optimum, helping to escape suboptimal regions.
    \item \textbf{Dynamic adjustment of $N_{\text{node}}$:} varying the number of candidate designs generated at each iteration based on convergence behavior or search progress.
    \item \textbf{Hybrid local-global search:} combining greedy search steps with stochastic exploration methods (e.g., random swaps or perturbations) to enhance search robustness.
\end{itemize}
A comprehensive investigation of these strategies is left for future work. The focus of the present study is to illustrate that applications of innovative GFS methods, and their integration with intelligent algorithms, already demonstrate diverse advantages and notable efficiencies.

Also, it is crucial to highlight that the process of selecting the optimal design differs from simply testing multiple samples and choosing the best among them. In this process, numerous designs are constructed based on a predetermined FIP, the best desing is selected, and this optimal design is then used to select a sample. This preserves the unbiasedness of the NHT estimator under the chosen design. Conversely, assessing multiple samples and selecting the most accurate estimation, possibly utilizing auxiliary variables, can lead to biased estimates.

\section{Summary and Suggestions}\label{Sec.Clu}
A graphical sampling framework, GFS, has been introduced in this study. GFS represents a novel gateway to finite population sampling, offering a fresh perspective that reshapes conventional thinking in sampling methodology. This innovative approach fosters creativity in developing new and efficient designs by enabling researchers to explore diverse sampling strategies. Through simple segment rearrangements, GFS can generate designs ranging from the baseline systematic construction of \citet{mad:49}—which often induces many zero SIP entries—to alternatives with broadly distributed design probabilities and a markedly denser SIP structure. Such breadth of exploration was previously unattainable using traditional mathematical algorithms.

Two notable components of this framework, Chaotic GFS and OGFS, illustrate how GFS can either generate new designs or identify optimal configurations based on specific criteria. Additionally, the inherent flexibility of GFS naturally invites integration with a wide variety of intelligent algorithms. Beyond the greedy best-first search algorithm employed in this study, methods such as genetic algorithms—where segments could be modeled as genes and configurations as chromosomes—as well as A*, simulated annealing, or particle swarm optimization, could be adapted to further enhance efficiency and precision.

Despite these strengths, a notable challenge emerges with the GFS approach as population size ($N$) grows large. Specifically, as the total number of possible samples ($T = 2^N$) increases exponentially, managing and preserving the full set of designs becomes computationally demanding. Nevertheless, given recent advancements in high-performance computing and the ongoing development of quantum computing technologies, these limitations are becoming increasingly manageable for small- to medium-sized populations. For larger populations and high-dimensional scenarios, further research into more sophisticated intelligent algorithms or compressed representation techniques could substantially enhance the scalability and practical applicability of GFS.

Looking ahead, a natural direction is to formalize how other standard designs (simple random, stratified, systematic, multi-stage) arise within this bar-based construction by specifying the appropriate constraints and update rules. A second priority is to investigate the conjectures formulated earlier on the limiting behavior of the chaotic updates (with and without a fixed-size constraint)—clarifying assumptions, establishing rates of concentration, and assessing robustness to implementation choices.

Given its flexibility in segment arrangement, the framework provides a unified mechanism to explore and optimize a wide array of design criteria—such as anticipated variance/efficiency, spatial spread, and distributional coverage—within a single construction. In this sense, the method offers promising avenues for generating diverse, efficient, and practically effective sampling designs.

\section*{Acknowledgments}
The author expresses sincere gratitude to three of his exceptionally talented students—Mehdi H. Moghadam, Mehdi Mohebbi, and Amir HosseiniNasab—for their invaluable contributions and insightful comments, particularly during the programming stage of this research. The author is also deeply grateful to Vincent Loonis\footnote{French National Institute of Statistics and Economic Studies, (INSEE)} for his generous guidance and valuable discussions, which significantly contributed to the simulation study. Additionally, the author warmly thanks Prof.~Yves Tillé for insightful discussions during the author's visit to the University of Neuchâtel, which helped significantly in refining the ideas presented in this paper.


\begin{thebibliography}{99}

\bibitem[Brewer and Hanif(1983)]{bre:han:83} Brewer, K.~R.~W. and Hanif, M. (1983). \newblock \emph{Sampling with Unequal Probabilities}. \newblock New York: Springer.

\bibitem[Deville and Till'e(2004)]{dev:til:04a} Deville, J.-C. and Till'e, Y. (2004). \newblock Efficient balanced sampling: The cube method. \newblock \emph{Biometrika}, \textbf{91}, 893--912.

\bibitem[Hansen and Hurwitz(1943)]{han:hur:43} Hansen, M.~H. and Hurwitz, W.~N. (1943). \newblock On the theory of sampling from finite populations. \newblock \emph{Annals of Mathematical Statistics}, \textbf{14}, 333--362.

\bibitem[Horvitz and Thompson(1952)]{hor:tho:52} Horvitz, D.~G. and Thompson, D.~J. (1952). \newblock A generalization of sampling without replacement from a finite universe. \newblock \emph{Journal of the American Statistical Association}, \textbf{47}, 663--685.

\bibitem[Loonis(2023)]{loo:23} Loonis, V. (2023). \newblock Constructing all determinantal sampling designs. \newblock \emph{Survey Methodology}, \textbf{49}, 411--441.

\bibitem[Loonis and Mary(2019)]{loonis2019determinantal} Loonis, V. and Mary, X. (2019). \newblock Determinantal sampling designs. \newblock \emph{Journal of Statistical Planning and Inference}, \textbf{199}, 60--88.

\bibitem[Madow(1949)]{mad:49} Madow, W.~G. (1949). \newblock On the theory of systematic sampling, II. \newblock \emph{Annals of Mathematical Statistics}, \textbf{20}, 333--354.

\bibitem[Narain(1951)]{nar:51} Narain, R.~D. (1951). \newblock On sampling without replacement with varying probabilities. \newblock \emph{Journal of the Indian Society of Agricultural Statistics}, \textbf{3}, 169--174.

\bibitem[Panahbehagh \emph{et~al.}(2024)]{graphical_sampling_2024} Panahbehagh, B., Mohebbi, M., HosseiniNasab, A., and Moghadam, M.~H. (2024). \newblock \emph{graphical-sampling} [Python package]. \newblock Available at \url{https://pypi.org/project/graphical-sampling/}.

\bibitem[S"arndal, Swensson and Wretman(1992)]{sar:swe:wre:92} S"arndal, C.-E., Swensson, B., and Wretman, J.~H. (1992). \newblock \emph{Model Assisted Survey Sampling}. \newblock New York: Springer.

\bibitem[Thompson and Seber(1996)]{thompson1996adaptive} Thompson, S.~K. and Seber, G.~A.~F. (1996). \newblock \emph{Adaptive Sampling}. \newblock Wiley Series in Probability and Statistics, Vol.~231. \newblock New York: Wiley. {ISBN: 978-0471965277}.

\bibitem[Till'e(2006)]{til:06} Till'e, Y. (2006). \newblock \emph{Sampling Algorithms}. \newblock New York: Springer.

\bibitem[Till'e(2020)]{till:2020} Till'e, Y. (2020). \newblock \emph{Sampling and Estimation From Finite Populations}. \newblock Hoboken: Wiley.

\end{thebibliography}
\end{document}